\documentclass[11pt]{amsart} 

\usepackage[utf8]{inputenc} 
\usepackage{amsmath}
\usepackage{amssymb}
\usepackage{mathtools}
\usepackage{mdwlist}
\usepackage{array}
\usepackage{url}

\def\ip#1#2{\<#1,#2\>}

\def\scr{\mathcal}
\def\hatscr#1{\hat{\scr #1}}
\def\e{\varepsilon}

\def\Conv{\operatorname{Conv}}

\def\bdy{\partial}

\def\ditemfirst#1#2{\begin{enumerate}\item \label{#1} #2\suspend{enumerate}}

\newtheorem{lem}{Lemma}
\newtheorem{prp}{Proposition}
\newtheorem{thm}{Theorem}

\usepackage{graphicx} 

\usepackage{array} 

\def\<{\langle}
\def\>{\rangle}

\def\R{\mathbb R}

\title[Domination and Proper Scoring Rules]{Necessary and Sufficient Conditions for Domination Results for Proper Scoring Rules}
\author{Alexander R. Pruss}

\begin{document}
\sloppy
\begin{abstract}
Scoring rules measure the deviation between a forecast, which assigns degrees of confidence to various events, and reality. Strictly proper scoring rules have the property
that for any forecast, the mathematical expectation of the score of a forecast $p$ by the lights of $p$  is strictly better than the
mathematical expectation of any other forecast $q$ by the lights of $p$. Forecasts need not satisfy the axioms of
the probability calculus, but Predd, et al.~(2009) have shown that given a finite sample space and any strictly proper \textit{additive and continuous}
scoring rule, the score for any forecast that does not satisfy the axioms of probability is strictly dominated by the score for some
probabilistically consistent forecast. Recently, this result has been extended to non-additive continuous scoring rules. In this paper,
a condition weaker than continuity is given that suffices for the result, and the condition is proved to be optimal.
\end{abstract}

\maketitle
\section{The main results}
Scoring rules measure the deviation between a forecast, which assigns degrees of confidence or credence to various events, and reality. Strictly proper scoring rules have the property
that for any forecast, the mathematical expectation of the score of a forecast $p$ by the lights of $p$  is strictly better than the
mathematical expectation of any other forecast $q$ by the lights of $p$. Forecasts need not satisfy the axioms of probability, but under
some conditions whose discussion is the main purpose of this paper, the score of a forecast that does not satisfy the axioms of probability
is strictly dominated by the score of a forecast by that does satisfy these axioms. This result has been interpreted by epistemologists
as supporting the idea that reasonable forecasts will always be probabilistically consistent (e.g., \cite{Joyce98}, \cite{GW}, \cite{Pettigrew16}).

To be precise, let $\Omega$ be a finite sample space, encoding the possible situations that the forecasts concern.
Let $\scr C$ be the set of all functions from the power set of $\Omega$ to the reals: these we call
credence functions. Let $\scr P$ be the subset of $\scr C$ which consists of the functions satisfying the axioms of probability.
Members of $\scr C$ can be thought of as forecasts regarding $\Omega$. An \textit{accuracy} (respectively, \textit{inaccuracy})
\textit{scoring rule} is a function $s$ from a set $\scr F \supseteq \scr P$ of credence function to $[-\infty,M]^\Omega$ (respectively, $[M,\infty]^\Omega$) for some finite $M$,
where $A^B$ is the set of functions from $B$ to $A$.
Then $s(c)(\omega)$ for $c\in\scr F$ measures
the accuracy of the forecast $c$ when we are in fact at $\omega\in\Omega$, with higher (respectively, lower) values being more accurate.
Since inaccuracy and accuracy scoring rules differ merely by a sign, we shall now assume that scoring rules are accuracy scoring rules,
translating results from the literature as needed.

Given a probability $p \in \scr P$ and an extended real function $f$ on $\Omega$, let $E_p f$ be the expected value with respect to $p$
defined in the following way to avoid multiplying infinity by zero:
$$
    E_p f = \sum_{\omega \in \Omega, p(\{\omega\})\ne 0} p(\{\omega\}) f(\omega).
$$
We then say that a scoring rule $s$ is \textit{proper} on $\scr F\supseteq \scr P$ provided that for every $p\in\scr P$ and
every $c\in\scr F$, we have $E_p s(p) \ge E_p s(c)$, that it is \textit{strictly proper} there provided the inequality
is always strict, and that it is \textit{quasi-strictly proper} there provided that it is proper and the inequality is strict when $p\in\scr P$ and
$c\in\scr F\backslash\scr P$.

Propriety captures the idea that if an agent adopts a probability function $p$ as their forecast, then by the agent's lights there can be no improvement in
the expected score from switching to a different forecast. Strict propriety captures the idea that an agent who has adopted a probability function $p$ as their forecast
will think other forecasts to be inferior. Proper and strictly proper scoring rules have been widely studied: for a few examples, see
\cite{DM}, \cite{GR}, \cite{Pettigrew16}, \cite{Predd}, \cite{WMC}.

A scoring rule is said to be \textit{additive} provided that $\scr F=\scr C$ and there is a collection of extended-real functions $(s_A)_{A\subseteq\Omega}$ on
$\R\times\{0,1\}$ such that for all $c\in\scr F$ and $\omega\in\Omega$:
$$
    s(c)(\omega) = \sum_{A \subseteq \Omega} s_A(c(A),1_A(\omega)).
$$

The set of probabilities $\scr P$ can be equipped with the topology resulting from its natural embedding $\psi$ in the product space $[0,1]^\Omega$,
where $\psi(p)(\omega) = p(\{\omega\})$. Thus, a sequence of probabilities $(p_n)$ converges to a probability $p$ just in case $p_n(\{\omega\})\to p(\{\omega\})$
for all $\omega\in\Omega$.

A scoring rule is \textit{probability-continuous} provided that the restriction of $s$ to $\scr P$ is a continuous function
to $[-\infty,M]^\Omega$ equipped with the Euclidean topology.

Say that a score $s(c_1)$ is \textit{weakly dominated} by a score $s(c_2)$ provided that $s(c_2)(\omega) \ge s(c_1)(\omega)$
for all $\omega\in\Omega$, and \textit{strictly dominated} if the inequality is strict.

Predd, et al.~\cite{Predd} showed that if $s$ is a probability-continuous additive strictly proper scoring rule, then for any non-probability $c$, there is a probability $p$ such
that $s(c)$ is strictly dominated by $s(p)$. In other words, any forecaster whose forecast fails to be a probability can find a forecast
that is a probability and that is strictly better no matter what. Recently, Pettigrew~\cite{PettigrewNew} announced that this
result holds without the assumption of additivity, merely assuming probability-continuity.  This proof was shown to have flaws~\cite{Nielsen},
but correct proofs were found by Nielsen~\cite{Nielsen} and Pruss~\cite{Pruss}. Nielsen's proof also extended the result to the quasi-proper case.
A philosophical upshot of these results is that we can get an argument in favor of probabilistic consistency in one's credence assignments under much
weaker conditions than the additivity assumed by Predd, et al.

However, the Pettigrew-Nielsen-Pruss theorem still assumes the continuity of the scoring rule. The purpose of the present paper is to identify what
is the weakest possible assumption on a strictly proper scoring rule as restricted to the probability functions that guarantees the strict domination property.

We need some definitions to state our main result.

Say that a score $s\in [-\infty,M]^\Omega$ is finite provided that $|s(\omega)|<\infty$ for all $\omega$.
Note that $\R^\Omega$ is an $n$-dimensional vector space.
Let $\ip\cdot\cdot$ be the usual scalar product on $\R^\Omega$: $\ip fg = \sum_{\omega} f(\omega)g(\omega)$
for any $f,g\in \R^\Omega$.
We say that a boundary point $z$ of a set $G\subseteq \R^\Omega$ is positive-facing provided that $G$ has a normal $v$ at $z$ all of whose
components are positive, i.e., there is a $v\in (0,\infty)^\Omega$ such that $\ip vw \le \ip vz$ for all $w\in G$. Write $\bdy^+ G$ for the set of
all the positive-facing boundary points.

Write $\Conv F$ for the convex hull of $F$, i.e., the union of all the line segments with
endpoints in $F$.  A set $A$ is dense in a set $B$ in a topological
space if every open set that intersects $B$ also intersects $A$ (e.g., the rational numbers are dense in the reals).

\begin{thm}\label{th:nec} Consider a proper scoring rule $s$ on $\scr P$. Then the following conditions are equivalent:
\begin{enumerate}
    \item[(a)] for every extension of $s$ to a quasi-strictly proper scoring rule $s':\scr C\to [-\infty,M]^\Omega$,
        if $c\in\scr C\backslash\scr P$, then there is a $p\in\scr P$ such that $s'(p)=s(p)$ strictly dominates $s'(c)$
    \item[(b)] either $E_p s(p)$ is infinite for some $p\in\scr P$ or both:
        \begin{enumerate}
            \item[(i)] for any sequence $(p_n)$ in $\scr P$ that converges to some probability function $p$ such
                that $s(p_n)$ is finite for all $n$ while $s(p)$ is not finite, we have $\lim_n E_{p_n} s(p_n) = E_p s(p)$, and
            \item[(ii)] if $F=s[\scr P]\cap\R^\Omega$ is the set of finite scores, then $F$ is dense in $\bdy^+ \Conv F$.
        \end{enumerate}
\end{enumerate}
\end{thm}

Combining the above with the following will yield a new proof of the Pettigrew-Nielsen-Pruss theorem.

\begin{prp}\label{prp:cont} Suppose that $s$ is quasi-strictly proper and continuous on $\scr C$. Then condition (b) of Theorem~\ref{th:nec}
    is fulfilled.
\end{prp}

It follows from Lemma~\ref{lem:limit}, below, that for any proper scoring rule $s$, the function $p\mapsto E_p s(p)$ is continuous on the
probabilities with finite score. Thus, in condition (b)(i) of Theorem~\ref{th:nec} we can drop the restriction that $s(p)$ is not finite.

The proofs of the Theorem and Proposition will be given in the next section.

Theorem~\ref{th:nec} becomes simpler in the special case where all the values of the proper scoring rule $s$ are finite, because
in that case condition (b)(i) is always satisfied, and (b)(ii) is necessary and sufficient for the domination condition (a).

To visualize geometrically what the crucial condition (b)(ii) says, identify our space $\Omega$ with
the set $\{1,2,...,n\}$. Then a probability $p$ can be thought of as a vector in $n$-dimensional Euclidean space $\R^n$ whose $k$th
coordinate is $p(\{k\})$, with all the coordinates non-negative and summing to one, and a score $s(p)$ can be thought of as an extended-real
``vector'' whose $k$th coordinate is $s(p)(k)$. By abuse of notation, in this geometrical explanation we won't distinguish probabilities
and the vectors corresponding to them, or scores and the vectors corresponding to them (we will be a little more careful when we get to proofs).
Then $F$ is the set of scores that lie in $\R^n$. The set $\bdy^+ \Conv F$ consists of
those points $z$ on the boundary of the convex hull of $F$ such that some ray starting at $z$ whose direction is positive in every coordinate
immediately leaves the convex hull of $F$. Condition (b)(ii) then says that any such $z$ is the limit of some sequence $s(p_i)$ of finite scores of
probabilities $p_i$.

When we form the convex hull of $F$, we are adding to $F$ various line segments to obtain a convex set (a set that contains the
line segment joining every pair of points in it). Doing this may increase the positive-facing boundary of $F$ to include some additional
line segments. Condition (b)(ii) then says that any point on any of these additional line segments has a point of $F$ arbitrarily
close to it. In some sense, this means that $F$ doesn't have any positive-facing open gaps.

We can illustrate this by describing a finite strictly proper scoring rule that does not satisfy (b)(ii).
Suppose $\Omega=\{1,2\}$,
so our probabilities and scores are identified with points in the plane $\R^2$. Given a probability $p$, i.e., a non-zero vector
with both coordinates non-negative, let $\theta(p)$ be the angle that $p$ makes with the $x$-axis. Then $\theta(p)$ ranges between
$0$ and $\pi/2$ radians. If $\theta(p)\le \pi/4$, let $s(p)$ be the point at angle $\theta(p)$ on the circle $T_1$ of radius $1$ with center $(1,0)$. Thus, $s(p)=(1+\cos \theta(p),\sin \theta(p))$. If $\pi/4<\theta(p)$, let $s(p)$ be the point at angle $\theta(p)$ on the circle
$T_2$ of radius $1$ with center $(0,1)$, so $s(p)=(\cos\theta(p),1+\sin\theta(p))$ (see Figure~\ref{fig:twocirc}, left). This is a strictly proper scoring
rule.\footnote{Here is a quick geometric sketch why. Consider a probability $p$ at angle $\theta(p)$. Let $L$ be a line through $s(p)$
tangent to the circle that $s(p)$ is on (Figure~\ref{fig:twocirc}, left).  The line $L$ is tangent to one of our circles at the point $s(p)$, which is at angle $\theta(p)$ on
the relevant circle. Thus, $L$ has angle $\theta(p)+\pi/2$ and is thus perpendicular
to the vector $p$, which has angle $\theta(p) \in [0,\pi/2]$. The line $L$ cuts the plane into two open half-spaces. Let $H$ be the half-space
that contains the origin (this half-space is below/left of $L$---see the Figure). Then $H$ is the set of points $z$ such that $\< p, z \> < \< p, s(p) \>$, where $\<\cdot,\cdot\>$ is the dot product. But $\< p, z \> = E_p z$ and
$\< p, s(p) \> = E_p s(p)$. Now, all possible scores lie on the arcs $AB$ and $CD$, with the point $C$ not being a possible score, and
so it is easy to see (considering the case where $s(p)$ is $B$ separately) that all the possible scores $s(q)$ for $q\ne p$ lie below
and/or to the left of the line $L$. Hence they all satisfy $E_p s(q) < E_p s(p)$, and we have strict propriety.} It's worth noting
if $T_1$ and $T_2$ were both the unit circle (so that the $\pi/4$ switchover was irrelevant), the resulting scoring rule would have been the
spherical scoring rule.

\begin{figure}\label{fig:twocirc}
\includegraphics[height=5cm]{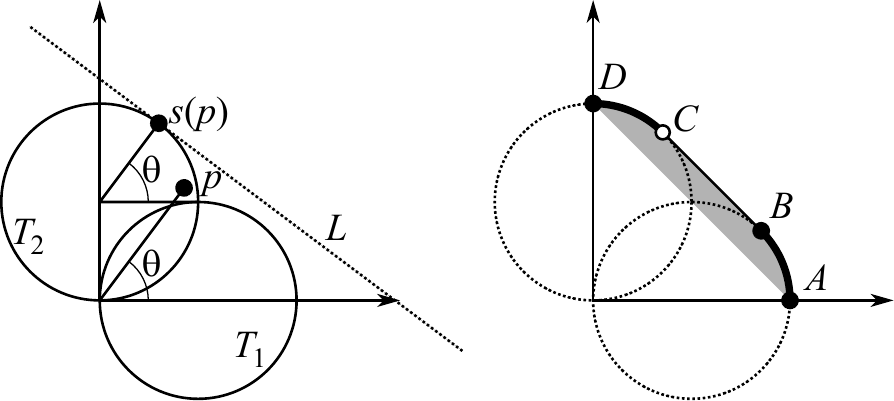}
\caption{Left: Calculating the score of $p=(3/7,4/7)$ with $\theta=\theta(p)=0.927>\pi/4$. Right: Thick lines and filled nodes indicate the set $F$. Shading indicates the convex hull. The set $\bdy^+ \Conv F$ consists of
of the thick lines, the thin line segment $BC$, and the filled and unfilled nodes.}
\end{figure}
In Figure~\ref{fig:twocirc} (right), the set $F$ of values (all finite) of $s$ consists of an arc $AB$ of $T_1$ from angle $0$ inclusive to $\pi/4$ inclusive together
with an arc $CD$ of $T_2$ from angle $\pi/4$ to $\pi/2$. The convex hull of $F$ consists
of the shaded region in the figure as well as some parts of the boundary of the shaded region\footnote{The
set $\Conv F$ includes all of its topological boundary except for the point $C$ and the interior of the line segment $BC$, but it doesn't
matter for the definition of $\bdy^+ \Conv F$ which parts of the boundary are members of $\Conv F$.}). The positive-facing boundary of $\Conv F$
consists of the arcs $AB$ and $CD$ as well as the line segment $BC$, together with all four endpoints. We can now see that $F$ is not dense
in $\bdy^+ \Conv F$, because $F$ has a gap consisting of the interior of the line segment $BC$, while $\Conv F$ has positive-facing boundary there.

Finally, we give a family of examples of proper scoring rules that satisfy (b)(i) and (b)(ii) but are not continuous.
The family will even include some scores that are finite everywhere on $\Omega$.
Let $\scr R$ be the set of probabilities that are \textit{regular} in the Bayesian sense, i.e., probability functions $p$ such
that $p(\{\omega\})>0$ for all $\omega$. Note that if $p$ is regular and $s$ is strictly proper, then we must have $E_p(s(p)) > -\infty$
or else we couldn't have $E_p(s(q)) < E_p(s(p))$ for a different probability $q$. Then by regularity of $p$ it follows that $s(p)$ is
finite.

Let $s$ be any continuous strictly proper everywhere-finite accuracy scoring rule.
Choose any $\alpha \in [-\infty,M]$
such that $\alpha \le s(p)(\omega)$ for all $p\in\scr P$ and $\omega\in\Omega$.
Define a scoring rule $s_\alpha$ as follows.
If either $p\notin \scr P$ or $p\in\scr R$, let $s_\alpha(p)=s(p)$. If $p\in\scr P\backslash\scr R$, then let $s_\alpha(p)(\omega)=s(p)(\omega)$
if $p(\{ \omega \})>0$, and $s_\alpha(p)(\omega)=\alpha$ if $p(\{ \omega \}) = 0$. In other words, we have tweaked $s$ so that
in the case where the forecast assigns zero probability to some outcome $\omega$, we assign $\alpha$ there.
It is easy to see that $s_\alpha$ is strictly proper because $E_p s_\alpha(p) = E_p s(p)$ and $E_p s_\alpha(q) \le E_q s(q)$ for all
$p,q\in\scr P$.

For a fixed $\omega$, the set of possible values of $s(p)(\omega)$ is contained in a finite interval, since a continuous function on
a compact set takes on a compact set of values, and $\scr P$ is compact. Thus we can choose $\alpha$ so that $\alpha < s(c)(\omega)$
for all $\omega$, letting $\alpha$ be $-\infty$ or a finite value as we wish. Then $s_\alpha$ will be discontinuous everywhere on
$\scr P\backslash\scr R$, since $s_\alpha$ is discontinuous wherever it differs from $s$. And if $\alpha$ is finite, then $s_\alpha$ will
be everywhere finite as well.

We now show that $s_\alpha$ satisfies the conditions of Theorem~\ref{th:nec}. First, $s$ satisfies condition (b) by Proposition~\ref{prp:cont},
and hence it satisfies condition (a) by the Theorem. Now let $s_\alpha'$ be a quasi-strictly proper extension of $s_\alpha$ to $\scr C$.
Define $s'(c)=s(c)$ for $c\in\scr P$ and $s'(c)=s_\alpha'(c)$ for $c\notin\scr P$. Then $s'$ is quasi-strictly proper, because for any
$c\notin\scr P$ and $p\in\scr P$ we have
$$
    E_p s'(c) = E_p s_\alpha'(c) < E_p s_\alpha(p) = E_p s(p) = E_p s'(p),
$$
where the first inequality follows from the quasi-strict propriety of $s_\alpha'$. Thus, because $s$ satisfies condition (a), for any $c\in\scr P$
there is a $p\in\scr P$ such that $s'(c)$ is strictly dominated by $s'(p)$. Since $s'$ is continuous on $\scr P$ and $\scr R$ is dense in $\scr P$,
we can choose the probability $p$ to be in $\scr R$. Then $s_\alpha'(c)=s'(c)$ will be strictly dominated by $s'(p)=s(p)=s_\alpha(p)$, and so
$s_\alpha$ will satisfy condition (a) of the Theorem, and hence also condition (b).

Note that in the case where $\alpha=-\infty$ the modified scoring rule $s_{-\infty}$ has some intuitive plausibility for measuring the accuracy of a credence assignment or forecast.
For if I assign credence zero to an outcome $\omega$ of our sample space, when in fact $\omega$ is the actual outcome, then I make
an error that in an important sense is infinitely bad, because no amount of Bayesian conditionalization on events with non-zero probability
will get me out of the error. The modified scoring rule $s_{-\infty}$ thus makes forecasts that are not regular be much more risky than a
an everywhere finite scoring rule $s$ does. I do not, of course, propose that scoring rules like $s_\alpha$ be generally adopted, but only want to
note that someone sufficiently attached to regularity might have some reason to do so.

\section{Proofs}
Given a probability function $p$, let $\hat p$ be the function from $\Omega$ to $[0,1]$ defined by $\hat p(\omega)=p(\{\omega\})$.
Define $\hatscr P = \{ \hat p : p\in\scr P \}$ and $\hatscr R = \{ \hat p : p \in\scr R \}$ (recall that $\scr R$ is the set of
regular probabilities, i.e., ones that are non-zero on every singleton). Thus, $\hatscr P$ is the set of
non-negative functions on $\Omega$ the sum of whose values is $1$, and $\hatscr R$ is the subset of the strictly positive ones.

Given $p\in \hatscr P$, let $\check p$ be the probability function such that $\check p(\{\omega\}) = p(\omega)$ for all $\omega\in\Omega$.

Given a scoring rule $s$ defined on $\scr P$, by abuse of notation let $s(p) = s(\check p)$ for $p\in\hatscr P$.

Given two functions $f$ and $g$ from $\Omega$ to the extended reals, say that $g$ strictly (weakly) dominates $f$ provided that $f<g$ ($f\le g$) everywhere.
The following is an easy fact.

Let $\cdot$ be a multiplication operation on the extended reals with the stipulation that $a \cdot 0 = 0\cdot a = 0$ for any $a$, finite or
not. With this stipulation, multiplication is upper semicontinuous on $[-\infty,\infty)\times [0,1]$, and continuous at all $(x,y)$ where $x$ is
finite or $y$ is positive. Now define the \textit{extended scalar product} on $X=[0,1]^\Omega \times [-\infty,\infty)^\Omega$ by
$$
    \ip fg = \sum_{\omega\in\Omega} f(\omega) \cdot g(\omega),
$$
using the above stipulation. The following properties are easy to check.

\begin{lem}\label{lem:dot} (a)~The extended scalar product is an upper semicontinuous function from $X$ to $[-\infty,\infty)$. Moreover, (b)~it is continuous
    at any $(f,g)$ such that either $f$ is strictly positive everywhere or $g$ is finite everywhere. Finally, (c)~for a fixed $f\in [0,1]^\Omega$,
    the function $g\mapsto \ip fg$ is continuous on $[-\infty,\infty)^\Omega$.
\end{lem}

Observe that $E_p f = \ip{\hat p}f$. This fact will allow us to go back and forth between probabilistic concepts and geometric concepts.

For a subset $F$ of $\R^\Omega$ and a vector $v\in\R^\Omega$, let
$$
    \sigma_F(v) = \sup_{z\in F} \ip vz
$$
be the support function of $F$ at $v$.

\begin{lem}\label{lem:geometrical}
    Assume $s$ is a proper scoring rule on $\scr P$ with $E_p s(p)$ finite for all $p$.
    Let $F$ be the set of finite scores. Suppose that for every convergent sequence $(p_n)$ of members of $\hatscr P$
    with $s(p_n)\in F$ for all $n$, we have $\lim_n \ip{p_n}{s(p_n)}=\ip p{s(p)}$, where $p=\lim_n p_n$.
    Then $\sigma_F(p) = \ip p{s(p)}$ for all $p \in\hatscr P$.
\end{lem}

\begin{proof}[Proof of Lemma~\ref{lem:geometrical}]
    First, suppose $s(p)\in F$.
    Then for every $z\in F$ we have $\ip pz \le \ip p{s(p)}$ by propriety.
    Since $s(p) \in F$, it follows that $\sigma_F(p) = \ip p{s(p)}$.

    Now suppose $s(p)$ is not finite. Let $(p_n)$ be a sequence in $\hatscr R$ converging
    to $p$. For $q\in\hatscr R$, the fact that $E_q s(q)$ is finite implies that $s(q)$ is
    finite, so $s(p_n)\in F$ for all $n$. Using compactness and passing to a subsequence if necessary,
    assume that $s(p_n)$ converges to some value $z\in [-\infty,M]^\Omega$. Then
$$
    \sigma_F(p) \ge \lim_n \ip p{s(p_n)}= \ip pz \ge \limsup_n \ip{p_n}{s(p_n)} = \ip p{s(p)},
$$
    where the relations follow respectively by definition of $\sigma_F$, the continuity of the extended scalar product for a
    fixed first argument (Lemma~\ref{lem:dot}(c)), the upper semicontinuity of the extended scalar product on $X$ (Lemma~\ref{lem:dot}(a)),
    and the assumptions of our present lemma.

    On the other hand:
$$
    \ip p{s(p)} \ge \ip p{s(q)}
$$
    by propriety for all $q\in\hatscr P$. Hence $\ip p{s(p)} \ge \sigma_F(p)$. Thus, $\sigma_F(p) = \ip p{s(p)}$.
\end{proof}

\begin{lem}\label{lem:limit}
Let $s$ be a proper scoring rule on $\scr P$.
Let $H=\{ p \in\hatscr P : s(p)\text{ is finite}\}$.
For for any $p$ in the closure $\bar H$ of $H$, the limit of $\ip q{s(q)}$ as $q$ tends to $p$ within $H$ exists.
If $s(p)$ is finite, the limit equals $\ip p{s(p)}$.
Finally, if $p_n$ is a sequence in $H$ converging to $p\in\bar H$ such that $\lim_n s(p_n)=r$, then $\ip pr = \lim_n \ip{p_n}{s(p_n)}$.
\end{lem}

\begin{proof}
Fix $p\in\bar H$.
To show that a sequence converges to some member of a compact set, it suffices to show that every convergent
subsequence of it converges to the same point. The set $[-\infty,M]$ is compact and $\ip q{s(q)} \in [-\infty,M]$ for all $q$.
Thus to show the existence of our limit, all we need to show is that if $(p_n)$ and $(p'_n)$ are two sequences in $H$
converging to $p$ with $\ip {p_n}{s(p_n)}\to L$ and $\ip {p'_n}{s(p'_n)}\to L'$, then $L=L'$. Moreover, if we can show this,
then letting $p'_n=p$ for all $n$, it will follow that $L=s(p)$ if $s(p)$ is finite.

Thus, suppose $(p_n)$ and $(p'_n)$ are two sequences in $H$ converging to $p$ with $\ip {p_n}{s(p_n)}$ and $\ip {p_n'}{s(p_n')}$
convergent. Passing to subsequences if necessary, assume that $s(p_n)$ and $s(p_n')$ converge respectively to $r$ and $r'$
in $[-\infty,M]^\Omega$.

I now claim that $\ip pr = \lim_n \ip{p_n}{s(p_n)}$. First, note that
$$
    \limsup_n \ip{p_n}{s(p_n)} \le \ip pr
$$
by Lemma~\ref{lem:dot}(a). Next observe that for any fixed $m$ we have
$$
    \ip{p_n}{s(p_n)} \ge \ip{p_n}{s(p_m)}
$$
by propriety.
Since $s(p_m)$ is finite as $p_m\in H$, by Lemma~\ref{lem:dot}(b) the right-hand-side converges to $\ip p{s(p_m)}$ as $n\to\infty$.
Thus:
$$
    \liminf_n \ip{p_n}{s(p_n)} \ge \ip p{s(p_m)}.
$$
Taking the limit as $m\to\infty$ and using Lemma~\ref{lem:dot}(c) we get
$$
    \liminf_n \ip{p_n}{s(p_n)} \ge \ip pr.
$$
Thus we have $\ip pr = \lim_n \ip{p_n}{s(p_n)}$, which is the final remark in the statement of our Lemma.

Thus:
\begin{align*}
    \ip p{r'} &= \lim_m \ip{p}{s(p_m')} \\
              &= \lim_m \lim_n \ip{p_n}{s(p_m')} \\
              &\le \lim_m \lim_n \ip{p_n}{s(p_n)} \\
              &=\lim_m \ip pr = \ip pr,
\end{align*}
where the first equality was due to Lemma~\ref{lem:dot}(c), the second due to Lemma~\ref{lem:dot}(b) and the finiteness of $s(p_m')$,
and the inequality was due to propriety. In exactly the same way, $\ip pr \le \ip p{r'}$. Thus, $\ip pr = \ip p{r'}$.
But $\ip p{r'} = \lim_n \ip{p_n'}{s(p_n')}$, just as we proved in the unprimed case. Thus, $\lim_n \ip{p_n}{s(p_n)} = \lim_n \ip{p_n'}{s(p_n')}$.
\end{proof}

\begin{lem}\label{lem:cone1}
Let $F$ be a nonempty subset of $[-\infty,\infty)^\Omega$.
Suppose $z_0\in [-\infty,\infty)^\Omega$ is such that $\ip{p}{z_0} < \sigma_F(p)$ for all $p\in\hatscr P$.
Then $z_0$ is strictly dominated by some member of $\Conv(F)$.
\end{lem}

\begin{proof}[Proof of Lemma~\ref{lem:cone1}]
Let $z_0$ be as in the statement of the Lemma.
Let $Q = \{ z \in \R^\Omega : \forall \omega(z(\omega) > z_0(\omega)) \}$ be the set of points of $\R^\Omega$ strictly
dominating $z_0$. We need to show that $Q \cap \Conv(F) \ne \varnothing$.

Suppose that $Q$ does not intersect $\Conv(F)$. Both $Q$ and $\Conv(F)$ are convex sets. Thus
by hyperplane separation, there is a non-zero $p\in\R^\Omega$ and an $\alpha\in\R$ such that
$\ip pz \ge \alpha$ for all $z\in Q$ and $\ip pz \le \alpha$ for all $z\in \Conv(F)$.

I claim that $p(\omega) \ge 0$ for all $\omega$. To see this, suppose $p(\omega_0) < 0$ for some $\omega_0\in \Omega$. Let $z$ be any member
of $Q$. Let $\beta = \ip pz$. Let $z'(\omega) = z(\omega)$ for $\omega\ne \omega_0$
and $z'(\omega_0) = z(\omega_0) + (\alpha-\beta-1)/p(\omega_0)$. Then $z'\in Q$ since $z\in Q$
while $\beta\ge\alpha$ and $p(\omega_0)<0$. Observe that
$$
    \ip p{z'} = \ip pz + (\alpha-\beta-1) = \alpha - 1,
$$
which is impossible as $\ip p{z'} \ge \alpha$ since $z'\in Q$.

Rescaling if necessary, we may assume that $\sum_\omega p(\omega)=1$ and hence $p\in\hatscr P$.

Since $z_0$ is on the boundary of $Q\subseteq [-\infty,\infty)^\Omega$, by the upper semicontinuity of the extended scalar product (Lemma~\ref{lem:dot}(a))
we have $\ip p{z_0} \ge \alpha$. Moreover, since $\ip pz\le \alpha$ for all $z\in\Conv(F)$, we must have $\sigma_F(p) \le \alpha$.
Thus, $\ip p{z_0} \ge \sigma_F(p)$, which contradicts the assumptions of the Lemma.
\end{proof}

Say that a vector $v$ is normal to a convex set $G$ at a point $z_1\in G$ provided that $\ip vz \le \ip v{z_1}$ for all $z\in G$.
The following Lemma is due to a \textit{MathOverflow} user (fedja 2021). Given a point $z \in \R^n$, let $Q_z$ be the positive orthant
$\R_+^n = (0,\infty)^n$ translated to put its vertex at $z$, i.e., $Q_z = \{ z+w : w \in \R_+^n \}$.

\begin{lem}\label{lem:convex}
    Fix $z_1 \in \R^n$.
    Let $G$ be a closed convex subset of $\mathbb R^n$ whose intersection with $Q_{z_1}$ is
    non-empty and bounded. Then there is a vector $v$ in the positive orthant $\R_+^n$ that is normal to $G$ at some point $z \in G\cap Q_{z_1}$.
\end{lem}

\begin{proof}
Translating $G$ as needed, assume without loss of generality that $z_1=0$ and $Q_{z_1}=\R_+^n$.

Let $f(x_1,\dots,x_n) = x_1\cdots x_n$ be the product-of-coordinates function. Let $U = G\cap \R_+^n$. Then $f$ is a continuous function
on the closure $\bar U$ of $U$ in $\R^n$.
Moreover, $\bar U$ is compact by the boundedness requirement, so $f$ attains a maximum on $\bar U$. This maximum is
not attained on the boundary of $\R_+^n$, i.e., on any point with a zero coordinate, since $f$ is zero at any such point,
while there is at least one point in $\bar U$ where $f$ is strictly positive (since $f$ is non-zero everywhere on $U\ne\varnothing$).
Hence, the maximum is attained at a point $z$ of $U$. That point cannot be an interior point (since then $z+(\e,\dots,\e)$
would be in $U$ and $f$ would be bigger there than at $z$), so it must be a boundary point of $U$ that isn't a boundary point of $\R_+^n$.
Thus, $z$ must be a boundary point of $G$. The basic normal cone condition for optimality tells us that the gradient
of $f$ must be normal to $U$ at $z$ for a differentiable $f$ to be optimal at $z$ over $U$ (see \cite[Theorem~6.12]{RW} in the
case of minima). Since $\R_+^n$ is a neighborhood of $z$, and the intersection of $U$ with that neighbhorhood
is the same as that of $G$ with it, it follows that the gradient of $f$ must be normal to $G$ at $z$. But the gradient of
$f$ on $\R_+^n$ always lies in $\R_+^n$, which completes the proof.
\end{proof}

Let $\|z\|_\infty = \max_{\omega\in \Omega} |z(w)|$ for $z\in [-\infty,\infty)^\Omega$ and let
$$
    B(z,r) = \{ z' : \| z'-z \|_\infty < r\}
$$
be the open ball of radius $r$ around a finite $z'$ in this norm.

\begin{lem}\label{lem:hole} Fix $p \in \hatscr R$ and $\alpha\in\R$, and let $K=\{ y \in [-\infty,\infty)^\Omega : \ip py \le \alpha \}$.
    Fix $\e>0$ and $x\in\R^\Omega$ such that $\ip px=\alpha$. Then there is a $\delta \in (0,\e)$ such that no point of $K\cap B(x,\delta)$ is weakly
    dominated by any point of $K\backslash B(x,\e)$.
\end{lem}

\begin{proof}[Proof of Lemma~\ref{lem:hole}]
Translating if necessary, without loss of generality assume $x=0$ and $\alpha=0$. Fix $y \in K\cap B(0,\delta)$. Suppose that $y$ is weakly dominated
by $z\in K$.

Then $-z(\omega)\le -y(\omega) \le \delta$ for all $\omega$ and $\sum_\omega z(\omega) p(\omega) \le 0$ as well as $\sum_\omega y(\omega) p(\omega) \le 0$.
Let $c = 1/\min_\omega p(\omega)$.
Then for any $\omega$:
$$
    z(\omega)\le -\frac{\sum_{\omega'\ne \omega} z(\omega') p(\omega')}{p(\omega)} \le -\frac{\sum_{\omega'\ne \omega} y(\omega') p(\omega')}{p(\omega)}
    \le \frac{\sum_{\omega'\ne \omega} \delta p(\omega')}{p(\omega)}
    \le c\delta.
$$
Moreover $-\delta \le z(\omega)$ and $c\ge 1$, so $\|z\|_\infty \le c\delta$. Thus, we have shown that if $y \in K\cap B(0,\delta)$ is weakly
dominated by $z \in K\backslash B(x,\e)$,
then $\e\ge c\delta$. Hence, if $\e>0$ is fixed, any choice of $\delta \in (0,c^{-1}\e)$ will complete the proof.
\end{proof}

\begin{proof}[Proof of Theorem~\ref{th:nec}]
If $E_p s(p)$ is infinite (i.e., equal to $-\infty$ since we are working with accuracy scores) for some probability function $p$, then $s$ has no extension to a quasi-strictly proper scoring
rule on $\scr C$, as no point $z\in [-\infty,\infty)^\Omega$ will be such that $E_p z < E_p s(p)$. Thus, we may assume for all our proofs that
$E_p s(p)$ is finite for all probabilities $p$, and hence that so is $\ip \hat p{s(\hat p)}$.

Without loss of generality, assume we have accuracy scores with ranges in $[-\infty,-1]$.

Recall that $E_p s(q) = \ip{\hat p}{s(\hat q)}$. If (i) fails, then there is a sequence $p_n\to p$ in $\hatscr P$ such that
$\ip{p_n}{s(p_n)}$ does not converge to $\ip p{s(p)}$ while $s(p_n)$ is finite and $s(p)$ is infinite. By Lemma~\ref{lem:limit}, $\lim_n \ip{p_n}{s(p_n)}=L$ exists
and cannot equal $\ip p{s(p)}$.
Passing to a subsequence if necessary, we may assume that $s(p_n)$ has a limit $r \in [-\infty,-1]^\Omega$, and then
\begin{equation}\label{eq:pr}
    \lim_n \ip{p_n}{s(p_n)} = \ip pr
\end{equation}
by the same Lemma.

Note that
$$
    \ip qr = \lim_n \ip q{s(p_n)} \le \ip q{s(q)}
$$
for all $q\in\hatscr P$ by Lemma~\ref{lem:dot}(c) and propriety. In particular, $L = \ip pr \le \ip p{s(p)}$. Thus, $L < \ip p{s(p)}$
since $\ip{p_n}{s(p_n)}$ does not converge to $\ip p{s(p)}$.  Further, $\ip p{s(p)}$ is negative since our accuracy scores are negative.
Choose $\alpha \in (1,L/\ip p{s(p)})$ so that
\begin{equation}\label{eq:alpha}
    \alpha \ip{p}{s(p)} > L = \ip pr
\end{equation}
Let $x = \alpha s(p)$.

Infinite scores cannot strictly dominate any score. Now I claim that $x$ is not weakly dominated by any finite score. For suppose
that $s(q)$ is finite. Then by Lemma~\ref{lem:dot}(b), propriety, \eqref{eq:pr} and \eqref{eq:alpha} we have:
$$
    \ip p{s(q)} = \lim_n \ip{p_n}{s(q)} \le \lim_n \ip{p_n}{s(p_n)} = \ip pr < \alpha \ip p{s(p)} = \ip px.
$$
And this implies that $x$ is not strictly dominated by $s(q)$.

On the other hand,
\begin{equation}
\label{eq:quasi}
    \ip qx = \alpha \ip q{s(p)} \le \alpha \ip q{s(q)} < \ip q{s(q)}
\end{equation}
for any $q\in\hatscr P$ since $\alpha>0$ and our scores are negative.

Now define $s'(c)=s(c)$ if $c\in\scr P$ and $s'(c)=x$ if $c\in\scr C\backslash\scr P$. Then
$s'$ is quasi-strictly proper by \eqref{eq:quasi}, but $s'(c)$ is not dominated by any score of a probability.
Hence condition (a) fails.

Now suppose (ii) fails. Thus there is a $z_0\in\bdy^+ \Conv F$ and $\e>0$ such that $F\cap B(z_0,\e)=\varnothing$.
Then $\Conv F$ has a normal in $(0,\infty)^\Omega$ at $z_0$. Rescaling if necessary, we can assume that normal is some $p\in\hatscr R$.
Let $K = \{ z \in [-\infty,\infty)^\Omega : \ip pz \le \ip p{z_0} \}$, which then contains $\Conv F$.
By Lemma~\ref{lem:hole}, there is a $\delta \in (0,\e)$ such that no point in $B(z_0,\delta)$ is strictly dominated by any
point in $K\backslash B(z_0,\e)$. Choose a point $z_1$ in $B(z_0,\delta)$ that
is strictly dominated by $z_0$. Note that $z_0$ is a limit of convex combinations of points of $F$. If $q$ is any member of $\hatscr P$ and
$u$ is any point of $F$, then $\ip qu \le \ip q{s(q)}$ by propriety. Thus, the same is true if $u$ is a convex combination of points of
$F$, and by Lemma~\ref{lem:dot}(c) also if $u$ is a limit of convex combinations of points of $F$.

Hence, $\ip q{z_0} \le \ip q{s(q)}$ for
all $q\in\hatscr P$. Since $z_1$ is strictly dominated by $z_0$, it follows that $\ip q{z_1} < \ip q{s(q)}$ for all $q\in\hatscr P$.

On the other hand, since $z_1$ is not strictly dominated by anything in $K\backslash B(z_0,\e)$, it's not strictly dominated by anything in
$F$ as $\Conv F\subseteq K$. Much as before, let $s'(c) = s(c)$ if $c$ is a probability function and $s'(c) = z_1$ if $c$ is not a probability function.
As before, we have strict quasi-propriety and yet no credence that isn't a probability is strictly $s'$-dominated by any probability function.

Now suppose that (i) and (ii) hold.
Let $s'$ be an extension of $s$ to a quasi-strictly proper scoring rule defined for all credences. Fix a non-probability credence $c$ and let
$z_0=s'(c)$. By (i) and Lemma~\ref{lem:geometrical}, $\ip p{s(p)}=\sigma_F(p)$ for all $p$. By Lemma~\ref{lem:cone1} and
quasi-strict propriety, $z_0$ is strictly dominated by some member of $\Conv(F)$.

Let $G$ be the closure of $\Conv(F)$ in $\R^\Omega$.
Then there is a point $z_1 \in (-\infty,\infty)^\Omega$ such that $z_0$ is strictly dominated by $z_1$ and $z_1$ is strictly dominated by some
member of $G$ (e.g., if $z_2$ is a point of $G$ that strictly dominates $z_0$, then let $z_1(\omega)=z_2(\omega)-1$ if $z_0(\omega)$ is infinite, and
$z_1(\omega)=(z_0(\omega)+z_2(\omega))/2$ otherwise).
Let $Q=Q_{z_1}$ be the set of points of $(-\infty,\infty)^\Omega$ that strictly dominate $z_1$. Then $Q\cap G$ is non-empty.
Note that every point $z$ of $F$ satisfies
\begin{equation}\label{eq:u}
\ip uz \le \ip u{s(u)}
\end{equation}
by propriety, where $u(\omega) = 1/|\Omega|$ for all $\omega$, and hence every point $z$ of $G$ satisfies \eqref{eq:u} as well
(a convex combination of points $z$ satisfying \eqref{eq:u} will satisfy it as well, and by Lemma~\ref{lem:dot}(c), so will every point in the closure of $\Conv F$).
Moreover, $\ip u{s(u)}$ is finite. The set of
points $z\in Q$ such that $\ip uz \le \ip u{s(u)}$ is bounded, and hence $Q\cap G$ is bounded.

Since $Q\cap G$ is bounded and non-empty, by Lemma~\ref{lem:convex} (letting $n=|\Omega|$ so that $\R^\Omega$ and $\R^n$ are isomorphic
as vector spaces), there is a $z_3\in Q\cap G$ such
that $z_3\in G$ has a normal in the positive orthant.
Thus, $z_3\in\bdy^+ G=\bdy^+ \Conv F$. By condition (ii), there are points of $F$ arbitrarily close to $z_2$. Since $z_2$ strictly
dominates $z_0$, so do points that are sufficiently close to $z_2$, and hence some point of $F$ strictly dominates $z_0$.
\end{proof}

\begin{proof}[Proof of Proposition~\ref{prp:cont}]
Suppose $s$ is continuous.
We must have $E_p s(p)$ finite for all probability functions $p$ or else quasi-strict propriety is impossible.
It follows that the score of any regular probability is finite.

By Lemma~\ref{lem:dot}(a), $p\mapsto \ip p{s(p)}$ is upper semicontinuous on $\hatscr P$ if the scoring rule is continuous.

Moreover, $\ip p{s(p)}  = \sup_{q\in\scr P} \ip p{s(q)}$. Since $\hatscr R$ is dense in $\hatscr P$ and $q\mapsto \ip p{s(q)}$ is
continuous for a fixed $p$ (using Lemma~\ref{lem:dot}(c)), it follows that $\sup_{q\in\hatscr P} \ip p{s(q)}=\sup_{q\in\hatscr R} \ip p{s(q)}$. Moreover,
since every score of a regular probability is finite, $p\mapsto \ip p{s(q)}$ is continuous by
Lemma~\ref{lem:dot}(b) for $q\in\hatscr R$. Thus, $p\mapsto \sup_{q\in\hatscr R} \ip p{s(q)}$ is the supremum of continuous functions and hence it
is lower semicontinuous at these points. (This observation is due to Nielsen [2021].)

Hence, $p\mapsto \ip p{s(p)}$ is continuous at every point of $\hatscr P$, and so we have (i).

It remains to prove (ii).
Assume first that $s$ is strictly proper.
Let $F$ be the set of finite scores of probabilities.
Then $F$ considered as a subset of $\R^\Omega$ is closed, since it is the intersection with $\R^\Omega$ of the set of
scores of probability functions, and the set of probability functions is compact while $s$ is continuous. Moreover, for any
$p \in \hatscr P$, we have $\ip p{s(p)} > \ip pz$ for all $z\in F\backslash \{ s(p) \}$. It follows that for any $p$, the only
point $z$ in the closed convex hull of $F$ such
that $\ip p{s(p)} = \ip pz$ is $s(p)$ itself, from which (ii) follows.

Now, suppose $s$ is merely quasi-strictly proper. Let $b:\scr C \to [-1,0]$ be any strictly proper continuous accuracy score, for instance
$-1$ plus the Brier score. Let $s_\e = s+\e b$ for any $\e>0$. Then $s_\e$ is a strictly proper continuous score, and (ii) must
hold for it. Let $F_\e = s_\e[\scr P] \cap \R^\Omega$.

Now, consider a point $z_0\in \bdy^+ \Conv F$. Fix $\e>0$. We will show that there is a point of $F$ within distance $\e$ of $z_0$.
Let $\delta=\e/4$. Thus there is a $p\in\hatscr R$ (since any vector in the positive orthant $(0,\infty)^\Omega$ can be rescaled to get a vector in $\hatscr R$) such
that $\ip pz \le \ip p{z_0}$ for all $z\in F$.
The point $z_0$ is a convex combination $z_0 = c_1 w_1 +\dots + c_n w_n$ of points of $F$, where $c_i > 0$ and $\sum_i c_i = 1$.
We then have to have $\ip p{w_i} = \ip p{z_0}$ for each $i$. Choose $p_i \in \hatscr P$ such that $w_i = s(p_i)$. Since $s_\delta$ is
continuous and $\hatscr R$ is dense in $\hatscr P$, for each $i$ choose $p_i'$ of $\hatscr R$ such
that $\| s_\delta(p_i')-s_\delta(p_i) \|_\infty \le \delta$. Then $\| w_i - s_\delta(p_i) \|_\infty \le 2 \delta$ since $s_\delta$ is never
more than $\delta$ away from $s$.

Let $z_0' = \sum_i c_i s_\delta(p'_i)$. Then, $\| z_0' - z_0 \|_\infty \le 2\delta$.
Note that $s_\delta(p_i') \in \bdy^+ F_\e$, since $p_i' \in (0,\infty)^\Omega$ is normal to $F_\e$ at $s_\delta(p_i')$ by the propriety of
$s_\delta$.
Applying the strictly proper case of our Proposition to $s_\delta$, we conclude that there is a $p\in\hatscr P$ such that $\| s_\delta(p) - z_0' \|_\infty \le \delta$
and $s_\delta(p) \in F_\e$. Then, $\| s(p) - z_0' \|_\infty \le 2\delta$, and so $\| s(p) - z_0 \|_\infty \le 4\delta = \e$, which completes the proof.\footnote{I am grateful to Michael Nielsen for many discussions and much encouragement, and to two
anonymous readers for a number of comments that helped with clarity, as well as for discovering a gap in my initial proofs.}
\end{proof}

\end{document}